\documentclass[]{interact}
\theoremstyle{plain}
\newtheorem{theorem}{Theorem}[section]

\newtheorem{corollary}[theorem]{Corollary}
\newtheorem{proposition}[theorem]{Proposition}
\theoremstyle{definition}
\newtheorem{definition}[theorem]{Definition}
\newtheorem{example}[theorem]{Example}
\newcommand{\ep}{\scriptsize\mbox{\textcircled{$\dagger$}}}

\newcommand{\cnn}{\mathbb{C}^{n\times n}}
\theoremstyle{remark}
\newtheorem{remark}{Remark}

\begin{document}

\title{G-Drazin inverse combined with inner  
inverse}

\author{
\name{G. Maharana\textsuperscript{a,1},  J. K. Sahoo\textsuperscript{a,2} and N\'estor Thome\textsuperscript{b}\thanks{CONTACT N.~Thome. Email: njthome@mat.upv.es }}
\affil{\textsuperscript{a} Department of Mathematics, BITS Pilani, K. K. Birla Goa Campus, Zuarinagar-403726, Goa, India. \textsuperscript{1}Email:p20180028@goa.bits-pilani.ac.in,~\textsuperscript{2}Email:jksahoo@goa.bits-pilani.ac.in\\ \textsuperscript{b}Instituto Universitario de Matem\'atica Multidisciplinar, Universitat Polit\`ecnica de Val\`encia, Camino de Vera, 14, Valencia, 46022, Spain}
}

\maketitle

\begin{abstract}
This paper introduces new classes of generalized inverses for square matrices named GD1, and the dual, called 1GD inverse. In addition, we discuss a few characterizations and representations of these inverses. The explicit expressions of these inverses have been established via core-nilpotent decomposition. Further, we introduce a binary relation for GD1 inverse and 1GD inverse, along with a few derived properties.
\end{abstract}


\section{Introduction and background}\label{sec1}
 M. P. Drazin \cite{draz} first introduced the Drazin inverse in associative rings and semigroups. Since then various applications of the Drazin inverse can be found in \cite{campbell2009generalized,campbell1976application,cline1980drazin,greville1974generalized,hartwig1978schur,nashed1992drazin,simeon1993drazin}. The Drazin inverse for bounded linear operators on complex Banach spaces was presented in \cite{caradus1978generalized,king1977note,lay1975spectral}. In \cite{mosic2018gdmp}, the gDMP inverse operator for a Hilbert space was discussed as an extension of the DMP inverse of a complex square matrix investigated in \cite{malik2014new}, and this appeared as an extension of the core inverse, which can be found in \cite{baksalary2010core}. Further, the DMP inverse has been generalized to a rectangular matrix, called W-weighted DMP inverse, which was established by Meng in \cite{meng2017dmp}. A deeper analysis of DMP inverse can be found in the cite \cite{MaGaSt} by  Ma, Gao and Stanimirovi\'{c}.  A recent generalized inverse, the CMP inverse (defined by $A^{\dagger}AA^DAA^{\dagger}$) of a square matrix $A$, was introduced by Mehdipour and Salemi \cite{MahSal}. A few characterizations and representations of the CMP inverse and its application are presented in \cite{ma2021characterizations}. The G-Drazin inverse for a square matrix was introduced in  \cite{wang2016partial} and then extended to operators in Banach spaces \cite{mosic2019weighted}.  The WgDMP inverse and dual WgDMP inverse for Hilbert space bounded linear operators were discussed in \cite{mosic2018weighted}. Representation of the G-Drazin inverse in a Banach algebra can be found in \cite{chen2019representation}. The extension of the notion of G-Drazin inverses to rectangular matrices along with a weight matrix was discussed by Coll et al. \cite{coll2018weighted}.  Generalized inverses for arbitrary index and square matrices have been discussed in \cite{malik2014new,manjunatha2014core,MahSal}. In \cite{mp1}, 1MP, and MP1 generalized inverses and their induced partial orders are discussed. Related to outer inverses, 2MP, MP2, and C2MP-inverses for rectangular matrices have been  studied in \cite{mp2}. The GDMP-inverse and its dual for square matrices were discussed in \cite{GDMP22}.

Motivated by work of \cite{mp1,GDMP22,mp2,D1}, in this paper we introduce and study a few characterizations of the GD1 inverse and its dual. The main contributions of this paper are the following:
\begin{enumerate}
\item[$\bullet$] Introduce two new classes of generalized inverses named as GD1 and the dual 1GD inverse.

\item[$\bullet$] A few characterizations and representations for these inverses are established.

\item[$\bullet$] Binary relations for these inverses introduced along with a few properties derived.
\end{enumerate}

 \subsection{Preliminaries}

In this subsection, we recall some notations, definitions and known results which will be used in proving our main results. Throughout this manuscript,  $\mathbb{C}^{m\times n }$ stands for the  set of  complex matrices of order  $m \times n$. The notation  $R(A),~N(A)$, and $A^*$  will denote the range space, the null space and the conjugate transpose of  a matrix $A \in \mathbb{C}^{m \times n}$ respectively. The smallest nonnegative integer $k$ such that $rank(A^k)=rank(A^{k+1})$  is known as the index of the matrix $A\in \mathbb{C}^{n\times n }$  and denoted  by $ind(A)$. The matrix $I$ will denote the identity of adequate size.

\subsection{A brief recapitulation of generalized inverses}
For a matrix $A\in\mathbb{C}^{m\times n}$, if a matrix $X$ satisfies $AXA=A$ then $X$ is called an inner inverse or generalized inverse of $A$. Further, if $XAX=X$, then $X$ is called outer inverse of $A$ and denoted by $A^{(2)}$. The set of all inner inverses of $A$ is denoted by $A\{1\}$ and an element of $A\{1\}$ is denoted as $A^{-}$.  Similarly, the set of all outer inverses of $A$ is denoted by $A\{2\}$. The Moore-Penrose inverse of a matrix $A$ is the unique matrix $X$ satisfying $AXA=A$, $XAX=X$ $(AX)^*=AX$, $(XA)^*=XA$ and denoted by $A^{\dagger}$. A matrix $X$ is called the Drazin  inverse of $A \in \mathbb{C}^{n\times n }$ (where $ind(A)=k$)  if it satisfies $XAX=X, XA=AX$, and $XA^{k+1} =A^k$, and we denote it by $A^D$.

Now, we will discuss a few composite generalized inverses which have been developed very recently. The 1MP inverse \cite{mp1}  of an arbitrary matrix $A$ is denoted as $A^{1MP}$ and defined by $A^{1MP}=A^{-}AA^{\dagger}$, where $A^-\in A\{1\}$. The dual of 1MP inverse, called MP1 inverse, is similarly denoted by $A^{MP1}$ and defined as $A^{MP1} = A^{\dagger}AA^{-}$. Further, the concept of 1MP and MP1 has extended to 2MP, MP2, and C2MP inverses as defined in what follows.
\begin{definition}\cite[Definition 2.1, Definition 3.1 and Definition 4.1]{mp2}
 Let  $A\in\mathbb{C}^{m\times n}$ and $A^{(2)}\in A\{2\}$.
 \begin{enumerate}
     \item[(a)] The 2MP inverse of $A$ is denoted by $A^{2MP}$ and defined as $A^{2MP}=A^{(2)}AA^{\dagger}$.
    \item[(b)] The MP2 inverse of $A$ is denoted by $A^{MP2}$ and defined as $A^{MP2}=A^{\dagger}AA^{(2)}$.
    \item[(c)] The C2MP inverse of $A$ is denoted by $A^{C2MP}$ and defined as $A^{C2MP}=A^{\dagger}AA^{(2)}AA^{\dagger}$.
 \end{enumerate}
\end{definition}
For arbitrary index square matrices, the combination of the Drazin inverse and the Moore-Penrose inverse, is renamed as DMP inverse \cite{malik2014new} and its dual called MPD inverse, which are defined below.
\begin{definition}\cite{malik2014new}
 Let  $A\in\mathbb{C}^{n\times n}$ with $ind(A)=k$.
 \begin{enumerate}
     \item[(a)] The DMP inverse of $A$ is denoted by $A^{D,\dagger}$ and defined as $A^{D,\dagger}=A^{D}AA^{\dagger}$.
    \item[(b)] The MPD inverse of $A$ is denoted by $A^{\dagger,D}$ and defined as $A^{\dagger,D}=A^{\dagger}AA^{D}$.
    \end{enumerate}
\end{definition}
 For a fixed inner inverse $A^-$, the 1D inverse  of a square matrix $A$ is denoted by $A^{-,D}$ and defined as $A^{-,D}=A^{-}AA^D$ \cite{D1}. Similarly, the dual of 1D inverse, called D1 inverse, is denoted by $A^{D,-}$ and defined as $A^{D,-} = A^DAA^{-}$.

Next, we recall the definition of G-Drazin inverse for square matrices.
\begin{definition} \cite{wang2016partial}
Let $A\in \mathbb{C}^{n \times n}$ and $ind(A)=k$. A matrix $X \in \mathbb{C}^{n\times n}$ is called a G-Drazin inverse of $A$ if it satisfies
\begin{equation}\label{GD_inverse}
 AXA=A, \quad XA^{k+1}=A^k, \quad \text{ and } \quad  A^{k+1}X=A^k.
\end{equation}
\end{definition}
We denote the set of all the G-Drazin inverses of $A$
by $A\{GD\}$ and an element of $A\{GD\}$ by $A^{GD}$. Note that G-Drazin inverse of an arbitrary index matrix always exists but need not be unique. In 2018, it was proved that the three equations in (\ref{GD_inverse}) can be reduced to the following two $AXA=A$ and $A^kX=XA^k$ \cite{coll2018weighted}. For further details on G-Drazin inverses, we refer the reader to \cite{weakdraz,coll2018weighted,mosic2019weighted,wang2016partial}.

 Drazin in \cite{Draz12} and \cite{Drazin2}  introduced $(b, c)$-inverses in the setting of a semigroup, which is a generalization of Moore-Penrose inverse. In order to state the matrix version of $(B,C)$-inverses we provide the following definition:
\begin{definition}\label{BC-invofA}
Let $A,~B,~C\in \mathbb{C}^{n \times n}$. A matrix  $X\in \mathbb{C}^{n \times n}$ is called  the $(B,C)$-inverse of $A$ if it satisfies
\begin{center}
$XAB=B,~CAX = C,~N(C)\subseteq N(X),~R(X) \subseteq R(B)$.
\end{center}
\end{definition}
It is known that if it exists, is unique.

Now, we recall one-sided partial orders and a few results on these relations.  
\begin{definition}\cite[Definition 6.3.1 and Corollary 6.3.10]{core-nil}\label{sharp-def}
  Let $A, B \in\mathbb{C}^{n\times n}$ be matrices of index at most 1. Then we will say that
  \begin{enumerate}
      \item[(i)] $A$ is below $B$ under the left sharp partial order ``${\#}\leq $" (we write $A{\#}\leq B$) if $A^2=AB$ and $R(A)\subseteq R(B)$. 
      \item[(ii)]  $A$ is below $B$ under the right sharp partial order ``$\leq{\#} $" (we write $ A\leq{\#} B$) if $A^2=BA$ and $R(A^*)\subseteq R(B^*)$. 
  \end{enumerate}
\end{definition}
\begin{definition}\cite{D1}
    Let $A, B\in\cnn$ with $ind(A)=k$.  We will say that $A$ is  below $B$ under the relation  $\leq ^ {D,-}$ if $A^{D,-}A = A^{D,-}B$ and $AA^{D,-} = BA^{D,-}$, and it is denoted by $A \leq^{D,-} B$.
\end{definition}
\begin{theorem}\cite{D1}\label{right-par}
Let $A, B \in\mathbb{C}^{n\times n}$ be matrices of index at most 1. If $A\leq ^{D,-}B$ then $A\leq \# B$.
\end{theorem}
\begin{corollary}\cite[Corollary 6.3.9]{core-nil}\label{corminus}
Let $A, B \in\mathbb{C}^{n\times n}$ be matrices of index at most 1. If $A\leq \# B$ then $AA^-=BA^-$ and $A^-A=A^-B$ for some $A^- \in A\{1\}$.
\end{corollary}

\section{GD1 inverse}
In this section, we first introduce the new class of generalized inverse, named GD1 inverse. From here onward, unless specified, we assume the matrix $A\in\cnn$ is of index $k$. Further, in the paper, we fix a 1-inverse and a GD-inverse; for those fixed inverses, we define just one GD1 inverse. However, it may change if we  range $A^-$ and $A^{GD}$ in their respective sets of generalized inverses.
\begin{definition}\label{gdd1}
 Let $A^-\in A\{1\}$ and  $A^{GD}\in A\{GD\}$.   The  GD1 inverse of $A$, associated with $A^{-}$ and $A^{GD}$, is denoted by $A^{GD1}$ and defined as $A^{GD1} =A^{GD}AA^{-}$.  \end{definition}
Further, we denote the set of all GD1 inverses by $A\{GD1\}$.
   \begin{example}\rm
Let $A=
\left[\begin{array}{cccc}
     1 & 0  &   0 & 0\\
     0  & 0   & 0 & 1\\
     0  & 0   & 0 & 0\\
      0  & 0   & 0 & 0
\end{array}\right]$ and we fix  $A^-=\left[\begin{array}{crrr}
     1 & 0  & 1 & 1\\
     0 &-1 &-1 &0\\
     0 &1 &1 &-1\\
     0 &1 &-1 &-1
    \end{array}\right]$. Clearly $ind(A)=2$ and we can calculate a GD inverse, $A^{GD}=
    \begin{bmatrix}
 1 & 0 & 0 & 0\\ 0 & 0& 0 & 1\\ 0 & 1& 0 & 0\\
 0 & 1 & 0 & 1\\
    \end{bmatrix}$. Thus, for these fixed $A^{-} \in A\{1\}$ and $A^{GD} \in A\{GD\}$, the GD1 inverse of $A$ is given by  
		$$A^{GD1}=A^{GD}AA^-=
		\left[\begin{array}{ccrr}
     1     &0    & 1     &1 \\
     0  & 0   & 0 & 0\\
     0     &1  &  -1    &-1\\
     0    & 1 &   -1    &-1\\
    \end{array}\right].
		$$ 
		Moreover, comparing to the known inverses, we observe that $A^{GD1}$ is not the same neither
 		$$
		A^{\dag}=
\left[\begin{array}{cccc}
     1 & 0  &   0 & 0\\
     0  & 0   & 0 & 0\\
     0  & 0   & 0 & 0\\
      0  & 1   & 0 & 0
\end{array}\right]$$ 
nor
$$ A^D=A^{D,\dagger}=A^{\dagger,D}=A^{\ep}=A^{\diamond}=A^{\textcircled{w}}=A^{\textcircled{w},\dagger}=
\left[\begin{array}{cccc}
     1 & 0  &   0 & 0\\
     0  & 0   & 0 & 0\\
     0  & 0   & 0 & 0\\
      0  & 0   & 0 & 0
\end{array}\right].
$$
\end{example}
The example above allows us to highlight that our new inverse is considerably different to the known ones in the literature.

Now, we derive some elemental properties of GD1 inverses.
 \begin{theorem}\label{theorem-1}
Let $A \in \mathbb{C}^{n \times n}$ with $A^-\in A\{1\}$ and  $A^{GD}\in A\{GD\}$ fixed. Then
\begin{enumerate}
\item[(i)] $A^{GD1} \in A\{1,2\}$. 
\item[(ii)]  $A^mA^{GD1} =A^mA^{-}$ and $A^{GD1}A^m = A^{GD}A^m$ for any  positive integer $m$.
\item[(iii)] $AA^{GD1}=P_{R(A),N(AA^{-})}$.
\item[(iv)]  $A^{GD1}A=P_{R(A^{GD}A),N(A)}$.
\end{enumerate}
\end{theorem}
\begin{proof}
(i) We have that $A A^{GD1} A = A A^{GD} A A^- A = A A^{GD} A = A$ and $A^{GD1} A A^{GD1} = A^{GD} A A^- A A^{GD} A A^-=
A^{GD} A A^{GD} A A^-= A^{GD} A A^-=A^{GD1}$.

(ii) If $m=1$ then $AA^{GD1}=AA^{GD}AA^-=AA^-$ and $A^{GD1}A=A^{GD}AA^-A=A^{GD}A$. Consider  $m\in\mathbb{N}$ with $m\geq 2$. Then
\begin{center}
    $A^mA^{GD1}=A^mA^{GD}AA^{-}=A^{m-1}AA^{GD}AA^{-} =A^mA^{-}$ and \\
  $A^{GD1}A^m= A^{GD}AA^{-}A^m = A^{GD}A A^{-} AA^{m-1} = A^{GD}A^m$.
\end{center}

(iii) From the Definition of GD1 inverse and (i), we obtain that $AA^{GD1}$ is idempotent. We can easily verify the range and null conditions from the following:
\begin{center}
  $AA^{GD1} = AA^{-}$, and $A^{GD1} \in A\{1\}$.
\end{center}
 (iv)  Clearly,  by (i), $A^{GD1} A $ is a projector. Similarly, from the following expressions
\begin{center}
$A^{GD1}A  = A^{GD}A A^-A = A^{GD}A$ and
		$A^{GD1} \in A\{1\}$,
\end{center}
we obtain  $R(A^{GD1}A)=R(A^{GD}A)$ and $N(A^{GD1}A)=N(A)$.
\end{proof}

Next result characterizes $GD1$ inverses from a geometrical point of view.
\begin{theorem}
Let $A \in \mathbb{C}^{n \times n}$ with  $A^-\in A\{1\}$ and  $A^{GD}\in A\{GD\}$  fixed. The $GD1$ inverse $A^{GD1}$ of $A$ is the unique matrix $X \in \mathbb{C}^{n \times n}$ that satisfies
\begin{equation}\label{eqngd1}
  AX = P_{R(A),N(AA^{-})} \qquad \text{ and } \qquad R(X)\subseteq R(A^{GD}A).
\end{equation}
\end{theorem}
\begin{proof}
Let $X=A^{GD1}$. Then by Theorem \ref{theorem-1}, we have $AX=A{A^{GD1}} = P_{R(A),N(AA^{-})}$. Clearly  $R(A^{GD1})=R(A^{GD}AA^-)\subseteq R(A^{GD}A)$. Hence, $A^{GD1}$ satisfies (\ref{eqngd1}).

Next we will show the uniqueness.
Suppose that there exist two solutions, say $X$ and $Y$, which satisfy the equation \eqref{eqngd1}. Then $A(X-Y)=P_{R(A),N(AA^{-})} -P_{R(A),N(AA^{-})}= 0$. Consequently, $R(X-Y) \subseteq N(A)= N(A^{GD}A)$. Further, from $R(X) \subseteq R(A^{GD}A)$ and $R(Y)\subseteq R(A^{GD}A)$,
we have $R(X-Y) \subseteq R(A^{GD}A)$. Thus $R(X-Y)\subseteq R(A^{GD}A) \cap N (A^{GD}A)=\{0\}$ because $A^{GD}A$ is a projector. Hence $X=Y$.
\end{proof}

Next, we characterize $GD1$ inverses by using information about the inner inverse or the GD inverse in two separate conditions.

\begin{theorem}\label{theorem-2}
Let $A \in \mathbb{C}^{n \times n}$ with  $A^-\in A\{1\}$ and  $A^{GD}\in A\{GD\}$  fixed. Then the following statements are equivalent:
\begin{enumerate}
     \item [(i)] $X = A^{GD1}$.
     \item [(ii)] $AX = AA^{-}$ and $R(X)=R(A^{GD}A)$.
     \item [(iii)] $A^{-}AX = A^{-}AA^{-}$ and $R(X)=R(A^{GD}A)$.
\end{enumerate}
\end{theorem}
\begin{proof}
(i) $\Rightarrow$ (ii) Let $X=A^{GD1}$. Then $AX=AA^{GD1} =AA^{GD}AA^{-} =AA^{-}$, $R(X)\subseteq R(A^{GD}A)$, and $R(A^{GD}A)=R(A^{GD}AA^-A)=R(XA)\subseteq R(X)$. \\
(ii) $\Rightarrow$ (i)  Let $R(X)=R(A^{GD}A)$. Then   $X=A^{GD}AY$ for some $Y\in\cnn$. Now
\begin{center}
$X=A^{GD}AY=A^{GD}AA^{GD}AY=A^{GD}AX=A^{GD}AA^{-}=A^{GD1}$. \end{center}
(ii) $\Rightarrow$ (iii) It is trivial.\\
(iii) $\Rightarrow$ (ii) Let $A^{-}AX = A^{-}AA^{-}$. Pre-multiplying by $A$, we obtain
$AX = AA^{-}$.
This completes the proof.
\end{proof}

In a similar manner, we can show the below theorem.
\begin{theorem}\label{theorem-3}
Let $A \in \mathbb{C}^{n \times n}$ with  $A^-\in A\{1\}$ and  $A^{GD}\in A\{GD\}$  fixed. Then the following statements are equivalent:
\begin{enumerate}
     \item [(i)] $X=A^{GD1}$.
     \item [(ii)] $XA=A^{GD}A$ and $N(X)=N(AA^{-})$.
     \item [(iii)] $XAA^{GD}=A^{GD}AA^{GD}$ and $N(X)=N(AA^{-})$.
\end{enumerate}
\end{theorem}
Next, we discuss an alternative characterization of the GD1 inverse. Up until now, we have focused on the inner inverse condition of GD1 inverses, and now we stress the outer inverse condition.

\begin{theorem}\label{thm4}
Let $A \in \mathbb{C}^{n \times n}$ with  $A^-\in A\{1\}$ and  $A^{GD}\in A\{GD\}$  fixed. Then the  following statements are equivalent:
\begin{enumerate}
     \item [(i)] $X=A^{GD1}$.
     \item [(ii)] $XAX=X,~R(X)=R(A^{GD}A)$, and $ N(X)=N(AA^{-})$.
     \item [(iii)] $XAX=X,~XA=A^{GD}A$, and $ AX=AA^{-}$.
     \end{enumerate}
\end{theorem}
\begin{proof}
(i) $\Rightarrow$ (ii) Let $X=A^{GD1}$. We have proved that $X \in A\{2\}$. From the Theorems \ref{theorem-2} and \ref{theorem-3}, it is clear that $R(X)=R(A^{GD}A)$ and $ N(X)=N(AA^{-})$. \\
(ii) $\Rightarrow$ (iii) From  $X=XAX$, we obtain  $X(I-AX) =0$. Thus, $R(I-AX) \subseteq N(X)=N(AA^{-})$ and hence  $AA^{-}(I-AX) =0$. Consequently, 
\begin{equation}\label{eq-q}
 AA^{-}=AA^{-}AX= AX.  
\end{equation} 
Let $R(X)=R(A^{GD}A)$. Then $X=A^{GD}AY$ for some $Y\in\cnn$. Now by using equation \eqref{eq-q}, we have 
\begin{center}  $XA=A^{GD}AYA=A^{GD}AA^{GD}AYA=A^{GD}AXA=A^{GD}AA^{-}A=A^{GD}A$.
\end{center}
(iii) $\Rightarrow$ (i) Clearly   $X=XAX=A^{GD}AX=A^{GD}AA^-=A^{GD1}$.
 \end{proof}

From a purely algebraic approach, we can provide the following result. The proof is similar  to that of Theorem
\ref{thm4} following the chain of implications (i) $\Rightarrow$ (ii)  $\Rightarrow$ (iii)  $\Rightarrow$ (iv) $\Rightarrow$ (i).

\begin{theorem}\label{thm5}
$A \in \mathbb{C}^{n \times n}$ with  $A^-\in A\{1\}$ and  $A^{GD}\in A\{GD\}$  fixed.  Then the following statements are equivalent:
\begin{enumerate}
     \item [(i)] $X = A^{GD1}$.
     \item [(ii)] $XAX =X,~ XA= A^{GD}A,~ AX = AA^{-}$, and $AXA = A$.
     \item [(iii)] $XAX =X,~ XA= A^{GD}A$, and $ AX = AA^{-}$.
     \item [(iv)] $A^{GD}AX =X,~~XA= A^{GD}A,~ AX = AA^{-}$, and $XAA^{-} = X$.
         \end{enumerate}
     \end{theorem}
	
Next results present GD1 inverses as the solution of a matrix rank equation.
\begin{theorem}
Let $A\in\cnn$ with $ind(A)=k$. 
\begin{enumerate}
    \item[(i)]  There exist idempotent matrices $M,N\in\cnn$ such that $$
A^kA^{-}M=0,~MA^k=0,~NA^kA^{-}=0,~ A^kN=0.
$$
\item[(ii)] There exists a matrix $X$ such that
$$
\textup{rank}\begin{bmatrix}
 A & I-M\\
 I-N & X
\end{bmatrix}=\textup{rank}(A).
$$
\end{enumerate}
Moreover,  $M=I-AA^{GD1}$, $N=I-A^{GD1}A$, and $X=A^{GD1}$.
\end{theorem}
\begin{proof}
 Let $M:=I-AA^{GD1}$  and  $N:=I-A^{GD1}A$. Since matrices $AA^{GD1}$ and $A^{GD1}A$ are idempotent, so both $M$ and $N$ are idempotent.
 
On the other hand, let $ind(A)=k\geq 1$. Then
\begin{eqnarray*}
A^kA^{-}M &=&A^kA^{-}(I-AA^{GD1})=A^kA^{-}-A^kA^{-}AA^{GD1}\\
&=&A^kA^{-}-A^{k-1}AA^{-}AA^{GD}AA^{-}=A^kA^{-}-A^{k-1}AA^{-}=0,
\end{eqnarray*}
and
\begin{equation*}
MA^k=(I-AA^{GD1})A^k=A^k-AA^{GD1}A^k=A^k-AA^{GD}AA^{-}A^k=A^k-A^k=0.
\end{equation*}
Further, we verify that
\begin{eqnarray*}
NA^kA^{-}=(I-A^{GD1}A)A^kA^{-}=A^kA^{-}-A^{GD}A^{k+1}A^-=A^kA^{-}-A^kA^{-}=0,
\end{eqnarray*}
and
\begin{equation*}
  A^kN=A^k(I-A^{GD1}A)=A^k-A^kA^{GD1}A=A^k-A^kA^{GD}AA^{-}A=A^k-A^k=0.
\end{equation*}
The case $k=0$ is trivial because $M=N=0$. 

Let $X=A^{GD1}$. Multiplying by block elemental matrices we get
\begin{eqnarray*}
    \textup{rank}\begin{bmatrix}
 A & I-M\\
 I-N & X
\end{bmatrix}&=&
\textup{rank} \begin{bmatrix}
 I & 0\\
 -A^{GD1} & I
\end{bmatrix} \begin{bmatrix}
 A & AA^{GD1}\\
 A^{GD1}A & A^{GD1}
\end{bmatrix}
\begin{bmatrix}
 I & -A^{GD1}\\
 0 & I
\end{bmatrix} \\
&=&\textup{rank}(A).
\end{eqnarray*}
\end{proof}

Next, we study the GD1 inverse as a $(B,C)$ inverse of $A$.
\begin{theorem}
Let $A\in\cnn$ with $A^-\in A\{1\}$ and  $A^{GD}\in A\{GD\}$ fixed. Then $A^{GD1}$ is the $(A^{GD}A,AA^{-})$ inverse of $A$.
\end{theorem}
\begin{proof}
Let $X=A^{GD1}$, $B=A^{GD}A$, and $C=AA^{-}$. Then
\begin{equation*}
    XAB=A^{GD1}AA^{GD}A=A^{GD}AA^{-}AA^{GD}A=A^{GD}AA^{GD}A=A^{GD}A=B
\end{equation*}
and
\begin{equation*}
CAX=AA^{-}AA^{GD1}=AA^{GD}AA^{-}=
AA^{-}=C.
\end{equation*}

 From $X=A^{GD}AA^-=BA^-$ and $X=A^{GD}AA^{-}=A^{GD}C$, we obtain $R(X)\subseteq R(B)$ and $N(C)\subseteq N(X)$. Now, by Definition \ref{BC-invofA}, we have that $A^{GD1}$ is the $(B,C)$ inverse of $A$. Hence, the proof is completed.
\end{proof}

\subsection{A binary relation based on GD1 inverses}
In view of the matrix relation defined for weighted Moore-Penrose inverse \cite{hernandez2013partial}, MP1 inverse \cite{mp1}, D1 inverse \cite{D1}, core inverse \cite{ferreyra2020some}, and G-Drazin inverse \cite{coll2018weighted,wang2016partial}, we introduce a binary relation for GD1 inverses.
\begin{definition}
Let $A,B \in\mathbb{C}^{n\times n}$. We will say that  $A$ is below $B$ under the relation $\leq ^{GD1}$ if $AA^{GD1}=BA^{GD1}$  and $A^{GD1}A=A^{GD1}B$ for some   $A^{GD1} \in  A\{GD1\}$.
\end{definition}

\begin{proposition}\label{gd1-d1}
Let $ A,B \in\mathbb{C}^{n\times n}$. If $A\leq ^{GD1}B$ then $A\leq ^{D,-}B$.
\end{proposition}
\begin{proof}
Assume that  $A\leq ^{GD1}B$. Then, $AA^{GD1}=BA^{GD1}$ and $A^{GD1}A=A^{GD1}B$ for some $A^{GD1} \in A\{GD1\}$. Hence, there exist $A^{GD} \in A\{GD\}$ and $A^- \in A\{1\}$ such that $AA^{GD}AA^-=BA^{GD}AA^-$ and $A^{GD}A=A^{GD}AA^-B$. Now, if we consider $A^{D,-}=A^DAA^-$, we get
\begin{eqnarray*}
A^{D,-}B&=&A^DAA^-B=A^DAA^{GD}AA^-B=A^DAA^{GD1}B=A^{D}AA^{GD1}A\\
&=&A^DAA^-A=A^{D,-}A,
\end{eqnarray*}
and
\begin{eqnarray*}
AA^{D,-}&=& AA^{GD}AA^-AA^{D,-}=AA^{GD1}AA^{D,-}=BA^{GD1}AA^{D,-}=BA^{GD}AA^DAA^-\\
&=&BA^{GD}A^{k+1}(A^D)^{k+1}AA^-=BA^k(A^D)^{k+1}AA^-=BA^D AA^-=BA^{D,-}.
\end{eqnarray*}
This completes the proof.
\end{proof}
From Proposition \ref{gd1-d1} and Theorem \ref{right-par}, we have the following result.
\begin{corollary}\label{corpo}
Let $A, B \in\mathbb{C}^{n\times n}$ be matrices of index at most 1. If $A\leq ^{GD1}B$ then $A\leq \# B$, where $\leq \#$ is the right sharp partial order.
\end{corollary}
It is clear that $\leq ^{GD1}$  is a reflexive relation. Since $\leq \#$ is a partial order, from Corollary \ref{corpo} we can conclude that $\leq ^{GD1}$ is an anti-symmetric relation. Next result completes the necessary information to obtain a partial order.

\begin{theorem}
  Let $A, B \in\mathbb{C}^{n\times n}$ be matrices of index at most 1. Then $A\leq {\#}B$ if and only $A\leq ^{GD1} B$.
  \end{theorem}
  \begin{proof}
  First notice that if $k=1$, from $A^2A^{GD}=A=A^{GD}A^2$ and recalling that $A^D A^2=A$, we obtain $AA^{GD}=A^{GD}A$.

	Now, let $A\leq {\#}B$. Then, by \cite[Definition 6.3.1]{core-nil}, we have $A^2=BA$. Moreover, by Corollary \ref{corminus}, there exists $A^- \in A\{1\}$ such that $AA^-=BA^-$ and $A^- A = A^- B$. For this $A^{-} \in A\{1\}$ and some GD inverse $A^{GD}$ of $A$, we define $A^{GD1}=A^{GD} A A^-$. Then,
  \begin{equation*}
  AA^{GD1}=AA^-=A^2A^{GD}A^{-}=BAA^{GD}A^{-}=BA^{GD}AA^-=BA^{GD1}
    \end{equation*}
and
    \begin{equation*}
    A^{GD1}A=A^{GD}AA^-A=A^{GD}AA^-B=A^{GD1}B.
        \end{equation*}
    Hence, $A\leq ^{GD1} B$. By Corollary \ref{corpo}, the proof is complete.
  \end{proof}
  \begin{remark}
      The relation $\leq ^{GD1}$ is a partial order on the set $\mathbb{C}^{GM}_n = \{A \in\mathbb{C}^{n\times n}:ind(A)\leq 1\}$.
  \end{remark}

Some characterizations of the relation $\leq^{GD1}$ can be presented by means of idempotents.
\begin{theorem}
Let $A,B\in\cnn$.  Then the following statements are equivalent:
\begin{enumerate}
     \item [(i)] $A\leq^{GD1}B$.
     \item [(ii)] $A =AA^-B=BA^{GD}A$ for some $A^- \in A\{1\}$ and some $A^{GD} \in A\{GD\}$.
     \item [(iii)] $A = AA^{GD1}B=BA^{GD1}A$ for some $A^{GD1} \in A\{GD1\}$.
     \item [(iv)] There exist idempotents $P,Q\in\cnn$ such that $R(P)=R(A)$, $N(P)=N(AA^-)$, $R(Q)=R(A^{GD}A)$, $N(Q)=N(A)$ and $A=PB=BQ$ for some $A^- \in A\{1\}$ and some $A^{GD} \in A\{GD\}$.
     \item [(v)] There exist idempotents $P,Q\in\cnn$ such that $R(P)=R(A)$, $N(P)=N(A^{GD1})$, $R(Q)=R(A^{GD1})$, $N(Q)=N(A)$ and $A=PB=BQ$ for some $A^{GD1} \in A\{GD1\}$.
\end{enumerate}
\end{theorem}

\begin{proof}
(i)$ \Rightarrow$ (ii) Let $A\leq ^{GD1}B$. Then, for some $A^{GD1} \in A\{GD1\}$ we have $AA^{GD1}=BA^{GD1}$ and $A^{GD1}A=A^{GD1}B$. Then,
\[
A=AA^{GD}A=AA^{GD}AA^-A=AA^{GD1}A=AA^{GD1}B=AA^{GD}AA^-B=AA^-B
\]
and
\[
A=AA^{GD}A=AA^{GD}AA^-A=AA^{GD1}A=BA^{GD1}A=BA^{GD}A.
\]
(ii) $\Rightarrow$ (iii) Let $A =AA^-B=BA^{GD}A$ for some $A^- \in A\{1\}$ and some $A^{GD} \in A\{GD\}$, and define $A^{GD1}:=A^{GD}AA^-$ using these generalized inverses. Then,
$A=AA^-B=AA^{GD}AA^-B=AA^{GD1}B$ and
$A=BA^{GD}A=BA^{GD}AA^-A=BA^{GD1}A$.\\
(iii) $\Rightarrow$ (i) Let $A = AA^{GD1}B=BA^{GD1}A$ for some $A^{GD1} \in A\{GD1\}$. Then,
\[
A^{GD1}A = A^{GD1}AA^{GD1}B=A^{GD1}B
\]
and
\[
AA^{GD1} = BA^{GD1}AA^{GD1} = BA^{GD1}.
\] \\
(i) $\Rightarrow$ (iv)
Assume that $AA^{GD1}=BA^{GD1}$ and $A^{GD1}A=A^{GD1}B$ for some $A^{GD1} \in A\{GD1\}$.
Let $P=AA^{GD1}$ and $Q=A^{GD1}A$.  From Theorem \ref{theorem-1}, it is clear that both $P$ and $Q$ are idempotent. From $A=AA^{GD}A=AA^{GD}AA^-A=AA^{GD1}A=PA$, we obtain $R(P)=R(A)$. Similarly, from $AA^-=AA^{GD}AA^-=AA^{GD}AA^{GD1}=AA^{GD}P$, we get $N(P)=N(AA^-)$. Now,
\[
A= AA^{GD}AA^-A=AA^{GD1}A=AA^{GD1}B=PB
\]
and 
\[
A= AA^{GD}AA^-A=AA^{GD1}A=BA^{GD1}A=BQ
\]
The range and null conditions  $R(Q)=R(A^{GD}A)$ and $N(Q)=N(A)$ follow from the below identities:
\[Q=A^{GD1}A=A^{GD}A \mbox{ and }A=BQ.
\]
(iv) $\Rightarrow$ (i) Assume that (iv) holds. For the existing $A^-$ and $A^{GD}$, we define $A^{GD1}:=A^{GD}AA^-$. Now,
by using that there exists only one projector with the same range and null spaces, by the hypothesis, it is easy to see that
$P=AA^{GD1}$ and $Q=A^{GD1}A$. Thus,
\begin{eqnarray*}
BA^{GD1}&=&BA^{GD}AA^-=BA^{GD}AA^-AA^-=BA^{GD1}AA^-=BQA^-=AA^-\\
&=&AA^{GD}AA^-=AA^{GD1},
\end{eqnarray*}
and
\begin{eqnarray*}
    A^{GD1}B&=&A^{GD}AA^-B=A^{GD}AA^{GD}AA^-B=A^{GD}AA^{GD1}B=A^{GD}PB=A^{GD}A\\
    &=&A^{GD}AA^-A=A^{GD1}A.
\end{eqnarray*}
(iv) $\Leftrightarrow$ (v) It is evident.
\end{proof}

\subsection{Computation of GD1 inverses by the core-nilpotent decomposition}
In this subsection, we present an explicit expression for GD1 inverses via the core-nilpotent decomposition. For $A\in\cnn$ with $ind(A) = k$ and rank($A^k$)$=s$, the core-nilpotent decomposition \cite{campbell2009generalized,core-nil} of $A$ is given by
\begin{equation}\label{cn}
   A=P\begin{pmatrix}
 C & 0\\0 & N
\end{pmatrix} P^{-1},
    \end{equation}
where  $C\in\mathbb{C}^{s\times s}$, $P\in\cnn$ both are   nonsingular, and $N\in\mathbb{C}^{(n-s)\times (n-s)}$ is nilpotent.  In view of \cite{wang2016partial}, a  G-Drazin inverse of $A$ can be written as
\begin{equation}\label{gd}
   A^{GD}=P\begin{pmatrix}
 C^{-1} & 0\\0 & N^-
\end{pmatrix} P^{-1}, \mbox{ where $N^{-}\in N\{1\}$.}
    \end{equation}
Using the above decompositions for $A$ and $A^{GD}$, we can find an explicit expression for a GD1 inverse as given below.

\begin{theorem}\label{gd1-deco}
Let  $A\in\cnn$ be written as in \eqref{cn}. Let $A^{GD}\in A\{GD\}$ written as in \eqref{gd} and let $A^{-}\in A\{1\}$. Then, the GD1 inverse of $A$ (given by these $A^-$ and $A^{G
D}$) can be expressed as
 \begin{equation*}
   A^{GD1}=P\begin{pmatrix}
 C^{-1} &  VN_r \\0 &
N^- N N^- + L N_{\ell}^- - N_{\ell}^- N_{\ell} L N_r
\end{pmatrix} P^{-1},
    \end{equation*}
  for arbitrary $V$ and $L$, where $N_{\ell}=I-N^-N$ and $N_r=I-NN^-$.
\end{theorem}

\begin{proof}
If we fix matrices $A^-$ and $A^{GD}$ as in the hypothesis, we can write
$$A^{GD1}=A^{GD}A A^- = P\begin{pmatrix}
 X & Y\\T & Z
\end{pmatrix} P^{-1}.
$$
 By comparing $A^{GD1}A=A^{GD}A$, we obtain
\begin{center}
    $X=C^{-1},~~YN=0,~~T=0,~~ZN=N^-N$.
\end{center}
Further, by using Theorem 1 (page 52) from
\cite{greville1974generalized}, the general solution of $ZN=N^{-}N$ is $Z=N^- N N^- + W (I-N N^-)$, for arbitrary $W$ and the general solution of $YN=0$ is given by $ Y=V(I-NN^-)$ for arbitrary $V$. Now, by Theorem \ref{theorem-1} (i), that is, by comparing $A^{GD1}AA^{GD1}=A^{GD1}$, we get $(I-N^- N) W (I-NN^-)=0$. By applying again \cite[Theorem 1 (page 52)]{greville1974generalized}, we obtain
$W=L-(I-N^-N)^- (I-N^-N) L (I-NN^-)(I-NN^-)^-$ for arbitrary $L$.  Finally, by substituting the expression of $W$ in that of $Z$ we get the result.
This completes the proof.
\end{proof}

Note that the matrix $Z=N^- N N^- + L(I-N^-N)^- - (I-N^-N)^-(I-N^-N) L (I-NN^-)$ in Theorem \ref{gd1-deco} is an inner inverse of $N$.  Thus a GD1 inverse of $A$ will be G-Drazin inverse of $A$ if and only if $Y=V(I-NN^-)=0$, which is stated in the following corollary.

Some particular cases of GD1 inverses are presented below.
\begin{corollary}
Let $A$, $A^{GD}$ and $A^{GD1}$ respectively as defined in Theorem \ref{gd1-deco}. Then
\begin{enumerate}
    \item[(i)] $A^{GD1}=A^{D,-}$ if and only if $N^- N N^- = W (N N^--I)$.
    \item[(ii)] $A^{GD1}=A^{GD}$ if and only if $V(I-NN^-)=0$.
    \item[(iii)] ${\cal A}\{GD\}\cap{\cal A}\{GD1\} =P\begin{pmatrix}
 C^{-1} & 0\\0 & N^- N N^- + W (I-N N^-)
\end{pmatrix} P^{-1}$,  where $W$ satisfies $(I-N^- N) W (I-NN^-)=0$.
\end{enumerate}
\end{corollary}

In view of the Theorem \ref{gd1-deco}, we can verify the following result.
\begin{corollary}
For a fixed $A^{-}\in A\{1\}$ and $A^{GD}\in A\{GD\}$, consider the decomposition's for $A$ and $A^{GD1}$ respectively as given in Theorem \ref{gd1-deco}. Then the following are equivalent:
\begin{enumerate}
    \item[(i)] $A\leq ^{GD1}B$.
    \item [(ii)] $B=P\begin{pmatrix}
 C & -CV(I-NN^-)B_4\\0& B_4
\end{pmatrix} P^{-1}$, where $B_4$ satisfies $(N^- N N^- + W (I-N N^-))B_4=N^{-}N$, $B_4(N^- N N^- + W (I-N N^-))=NN^- +N W (I-N N^-)$, and $W$ satisfies $(I-N^- N) W (I-NN^-)=0$.
\end{enumerate}
\end{corollary}

\section{1GD inverse}
In this section, we introduce 1GD inverse for square matrices along with a few characterizations of
these inverses. Since the methodology of the proofs is similar to those for GD1 inverses, the proofs will
not be included.
\begin{definition}
Let $A^-\in A\{1\}$ and $A^{GD}\in A\{GD\}$. The 1GD inverse of $A$, associated with $A^-$ and $A^{GD}$,  is denoted by $A^{1GD}$ and defined as $A^{1GD}= A^{-}AA^{GD}$.
\end{definition}
  \begin{example}\rm
Let $A=
    \begin{bmatrix}
     1 & 1  &   1 & 1\\
     0  & 0   & 0 & 1\\
     0  & 0   & 1 & 0\\
      0  & 0   & 0 & 0\\
    \end{bmatrix}$ and we fix  $A^{-}=\left[\begin{array}{crrc}
     1 & -1  &   -1 & 0\\
     0  & 0   & 0 & 0\\
     0  & 0   & 1 & 0\\
     0  & 1   & 0 & 0\\
    \end{array}\right]$. Clearly $ind(A)=2$ and we can calculate a GD inverse, $A^{GD}=
    \left[\begin{array}{crrc}
 1 & -1 & -1 & 2\\
0 & 0& 0 & 0\\
0 & 0& 1 & 0\\
 0 & 1 & 0 & 0\\
    \end{array}\right]$. Thus, for these fixed $A^-\in A\{1\}$ and $A^{GD}\in A\{GD\}$, the 1GD inverse of $A$ is given by $A^{1GD}=A^{-}AA^{GD}=\left[\begin{array}{crrc}
     1 & -1  &   -1 & 2\\
     0  & 0   & 0 & 0\\
     0  & 0   & 1 & 0\\
     0  & 1   & 0 & 1\\
    \end{array}\right]$.
\end{example}
\begin{theorem}
Let $A\in\mathbb{C}^{n\times n}$ with $A^-\in A\{1\}$ and $A^{GD}\in A\{GD\}$ fixed. Then
\begin{enumerate}
    \item[(i)] $A^{1GD}\in A\{1,2\}$.
    \item[(ii)] $A^mA^{1GD} =A^mA^{GD} $ and $ A^{1GD}A^m = A^{-}A^m$ for any positive integer $m$.
     \item[(iii)] $AA^{1GD}=P_{R(A),N(AA^{GD})}$.
    \item[(iv)] $A^{1GD}A=P_{R(A^{-}A),N(A)}$.
    \end{enumerate}
\end{theorem}
\begin{theorem}
Let $A \in\cnn$ with $A^-\in A\{1\}$ and $A^{GD}\in A\{GD\}$ fixed. The 1GD inverse, $A^{1GD}$ of $A$ is the unique matrix that satisfies
\begin{equation*}
  AX = P_{R(A),N(AA^{GD})} \quad  \mbox{ and } \quad R(X)\subseteq R(A^{-}A).
\end{equation*}
 \end{theorem}
\begin{theorem}
Let $A\in\cnn$ with $A^-\in A\{1\}$ and $A^{GD}\in A\{GD\}$ fixed. Then the following statements are equivalent:
\begin{enumerate}
     \item[(i)] $X = A^{1GD}$.
     \item[(ii)] $AX = AA^{GD}$ and $R(X)=R(A^{-}A)$.
     \item[(iii)] $A^{GD}AX = A^{GD}AA^{GD}$ and $R(X)=R(A^{-}A)$.
\end{enumerate}
\end{theorem}
\begin{theorem}
Let $A\in\cnn$ with $A^-\in A\{1\}$ and $A^{GD}\in A\{GD\}$ fixed.  Then the following statements are equivalent:
\begin{enumerate}
     \item[(i)] $X=A^{1GD}$.
     \item[(ii)] $XA=A^{-}A$ and $N(X)=N(AA^{GD})$.
     \item[(iii)] $XAA^{-}=A^{-}AA^{-}$ and $N(X)=N(AA^{GD})$.
\end{enumerate}
\end{theorem}
\begin{theorem}
Let $A\in\cnn$ with $A^-\in A\{1\}$ and $A^{GD}\in A\{GD\}$ fixed.  Then the following statements are equivalent:
\begin{enumerate}
     \item[(i)] $X=A^{1GD}$.
     \item[(ii)] $XAX=X,~R(X)=R(A^{-}A)$, and $ N(X)=N(AA^{GD})$.
      \item[(iii)] $XAX=X,~XA=A^{-}A$, and $ AX=AA^{GD}$.
   \end{enumerate}
\end{theorem}
\begin{theorem}
Let $A\in\cnn$ with $A^-\in A\{1\}$ and $A^{GD}\in A\{GD\}$ fixed.  Then the following statements are equivalent:
\begin{enumerate}
     \item[(i)] $X = A^{1GD}$.
     \item[(ii)] $XAX =X,~XA= A^{-}A,~ AX = AA^{GD}$, and $AXA = A$.
     \item[(iii)] $XAX =X,~ XA= A^{-}A$, and $AX = AA^{GD}$
     \item [(iv)] $A^{-}AX =X,~XA= A^{-}A,~ AX = AA^{GD}$, and $XAA^{GD} = X$.
    \end{enumerate}
\end{theorem}
\begin{theorem}
Let $M,N\in\cnn$ and $A\in\cnn$ with $ind(A)=k$.
\begin{enumerate}
    \item[(i)] There exist idempotent matrices $M$ and $N$ such that $$A^{-}A^kM=0,~MA^k=0,~  NA^-A^k=0,~ A^kN=0.$$
\item[(ii)] There exists a matrix $X$ such that
$$
\textup{rank}\begin{bmatrix}
 A & I-M\\
 I-N & X
\end{bmatrix}=\textup{rank}(A).
$$
\end{enumerate}
Moreover, $M=I-AA^{1GD}$, $N=I-A^{1GD}A$, and $X=A^{1GD}$.
\end{theorem}
\begin{theorem}
Let $A\in\cnn$ with $A^-\in A\{1\}$ and $A^{GD}\in A\{GD\}$ fixed. Then $A^{1GD}$ is the $(A^{-}A,AA^{GD})$ inverse of $A$.
\end{theorem}

\subsection{A relation based on 1GD inverses}
\begin{definition}
Let $A, B \in\mathbb{C}^{n\times n}$. We will say that $A$ is below $B$ under the relation $\leq ^{1GD}$ if $AA^{1GD}=BA^{1GD}$  and $A^{1GD}A=A^{1GD}B$ for some $A^{1GD}\in A\{1GD\}$.
\end{definition}

\begin{proposition}
Let $A, B \in\mathbb{C}^{n\times n}$.   $A\leq ^{1GD}B$  then $A\leq ^{-,D}B$.
\end{proposition}
\begin{corollary}
Let $A, B \in\mathbb{C}^{n\times n}$  be matrices of index at most 1. If  $A\leq ^{1GD}B$  then  $A{\#}\leq B$ , where ${\#}\leq $
is the left sharp partial order.
\end{corollary}

\begin{theorem}
Let $A, B \in\mathbb{C}^{n\times n}$  be matrices of index at most 1. Then $A\leq ^{1GD}B$ if and only if $A{\#}\leq B$.
\end{theorem}

\begin{remark}
 The relation  $\leq ^{1GD}$ is a partial order on the set $\mathbb{C}^{GM}_n=\{A\in \cnn:ind(A)\leq 1\}$.
\end{remark}

\begin{theorem}
Let $A,B\in\cnn$. Then the following statements are equivalent:
\begin{enumerate}
     \item[(i)] $A<^{1GD}B$.
     \item[(ii)] $A =BA^{-}A=AA^{GD}B$ for some $A^-\in A\{1\}$ and some $A^{GD}\in A\{GD\}$.
     \item[(iii)] $A = BA^{1GD}A=AA^{1GD}B$ for some  $A^{1GD}\in A\{1GD\}$.
     \item[(iv)] there exist idempotents $P,Q\in\cnn$ such that $R(P)=R(A)$, $N(P)=N(AA^{GD})$, $R(Q)=R(A^{-}A)$, $N(Q)=N(A)$ and $A=PB=BQ$ for some $A^-\in A\{1\}$ and some $A^{GD}\in A\{GD\}$.
     \item[(v)] there exist idempotents $P,Q\in\cnn$ such that $R(P)=R(A)$, $N(P)=N(A^{1GD})$, $R(Q)=R(A^{1GD})$, $N(Q)=N(A)$ and $A=PB=BQ$ for some $A^{1GD}\in A\{1GD\}$.
\end{enumerate}
\end{theorem}

\begin{theorem}\label{1gd-dec}
Let $A\in\cnn$and $A^{GD}\in A\{GD\}$ be respectively written as in \ref{cn} and \ref{gd}. For $A^-\in A\{1\}$, the 1GD inverses of $A$ can be expressed as
\begin{equation*}
 A^{1GD}=P\begin{pmatrix}
C^{-1} & ~~0\\
  N_{\ell} V &  ~~N^- N N^- + N_{\ell} L - N_{\ell} LN_r N_r^-
 \end{pmatrix}P^{-1},
 \end{equation*}
 for arbitrary $V$ and $L$, where $N_{\ell}=I-N^-N$ and $N_r=NN^--I$.
\end{theorem}

\begin{corollary}
Let $A$, $A^{GD}$ and $A^{1GD}$ be respectively as defined in Theorem \ref{1gd-dec}. Then
\begin{enumerate}
     \item[(i)] $A^{1GD}=A^{-,D}$ if and only if $N^-NN^-=(N^-N-I)L+(I-N^-N)L(NN^--I) (NN^--I)^-$.
     \item[(ii)] $A^{1GD}=A^{GD}$  if and only if $(I-N^-N)V=0$.
     \item[(iii)] ${\cal A}\{GD\}\cap{\cal A}\{1GD\} =P\begin{pmatrix}
 C^{-1} & 0\\0 & N^- N N^- +  (I-N^-N)W
\end{pmatrix} P^{-1}$,  where $W$ satisfies $(I-N^- N) W (NN^--I)=0$.
     \end{enumerate}
\end{corollary}
\begin{theorem}
For a fixed $A^-\in A\{1\}$ and $A^{GD}\in A\{GD\}$, consider the decomposition's for $A$ and $A^{1GD}$ respectively as given in Theorem \ref{1gd-dec}. Then the following are equivalent:
\begin{enumerate}
     \item[(i)] $A<^{1GD}B$.
     \item[(ii)] $B =P\begin{pmatrix}
 C & ~0\\-B_4(I-N^-N)VC & ~B_4
\end{pmatrix} P^{-1}$, where $B_4$ satisfies the equalities $B_4(N^- N N^- + (I-N^- N)W)=NN^{-}$, $(N^- N N^- + (I-N^-N)W)B_4=N^-N +(I-N^-N)WN$, and $W$ satisfies $(I-N^- N) W (NN^--I)=0$.
     \end{enumerate}
\end{theorem}

\section*{Disclosure statement}
The authors declare that, they have no relevant or material financial interests that relate to the research described in this paper.

\section*{Funding}

The second author is partially supported by Science \& Engineering Research Board (SERB), Govt. of India (Grant ID SUR/2022/004357). The third author was partially supported by Universidad Nacional de R\'{\i}o Cuarto (Grant PPI 18/C559), Universidad Nacional de La Pampa, Facultad de Ingenier\'ia (Grant Resol. Nro. 135/19), Universidad Nacional del Sur (Grant PGI 24/L108), and  Ministerio de Ciencia e Innovaci\'on of Spain (Grant Red de Excelencia RED2022-134176-T).


\begin{thebibliography}{10}
\bibitem{baksalary2010core}
O.~M. Baksalary and G.~Trenkler.
\newblock Core inverse of matrices.
\newblock {\em Linear Multilinear Algebra}, 58(5-6):681--697, 2010.

\bibitem{campbell2009generalized}
S.~L. Campbell and C.~D. Meyer.
\newblock {\em Generalized inverses of linear transformations}, volume~56 of
  {\em Classics in Applied Mathematics}.
\newblock SIAM, Philadelphia, PA, 2009.

\bibitem{weakdraz}
S.~L. Campbell and C.~D. Meyer, Jr.
\newblock Weak {D}razin inverses.
\newblock {\em Linear Algebra Appl.}, 20(2):167--178, 1978.

\bibitem{campbell1976application}
S.~L. Campbell, C.~D. Meyer, Jr., and N.~J. Rose.
\newblock Applications of the {D}razin inverse to linear systems of
  differential equations with singular constant coefficients.
\newblock {\em SIAM J. Appl. Math.}, 31(3):411--425, 1976.

\bibitem{caradus1978generalized}
S.~Caradus.
\newblock Generalized inverses and operator theory, queen’s papers in pure
  and appl.
\newblock {\em Math., Queen’s Univ., Kingston}, 1978.

\bibitem{cline1980drazin}
R.~E. Cline and T.~Greville.
\newblock A {D}razin inverse for rectangular matrices.
\newblock {\em Linear Algebra and its Applications}, 29:53--62, 1980.

\bibitem{coll2018weighted}
C.~Coll, M.~Lattanzi, and N.~Thome.
\newblock Weighted {G}-{D}razin inverses and a new pre-order on rectangular
  matrices.
\newblock {\em Appl. Math. Comput.}, 317:12--24, 2018.

\bibitem{draz}
M.~P. Drazin.
\newblock Pseudo-inverses in associative rings and semigroups.
\newblock {\em Amer. Math. Monthly}, 65:506--514, 1958.

\bibitem{Draz12}
M.~P. Drazin.
\newblock A class of outer generalized inverses.
\newblock {\em Linear Algebra Appl.}, 436(7):1909--1923, 2012.

\bibitem{Drazin2}
M.~P. Drazin.
\newblock Left and right generalized inverses.
\newblock {\em Linear Algebra Appl.}, 510:64--78, 2016.

\bibitem{ferreyra2020some}
D.~Ferreyra and S.~B. Malik.
\newblock Some new results on the core partial order.
\newblock {\em Linear and Multilinear Algebra}, pages 1--17, 2020.

\bibitem{greville1974generalized}
T.~N. Greville and A.~Ben-Israel.
\newblock {\em Generalized inverses: theory and applications}.
\newblock Wiley New York, 1974.

\bibitem{hartwig1978schur}
R.~E. Hartwig.
\newblock Schur's theorem and the {D}razin inverse.
\newblock {\em Pacific J. Math.}, 78(1):133--138, 1978.

\bibitem{hernandez2013partial}
A.~Hern\'{a}ndez, M.~Lattanzi, and N.~Thome.
\newblock On a partial order defined by the weighted {M}oore-{P}enrose inverse.
\newblock {\em Appl. Math. Comput.}, 219(14):7310--7318, 2013.

\bibitem{mp1}
M.~V. Hern\'{a}ndez, M.~B. Lattanzi, and N.~Thome.
\newblock From projectors to 1{MP} and {MP}1 generalized inverses and their
  induced partial orders.
\newblock {\em Rev. R. Acad. Cienc. Exactas F\'{\i}s. Nat. Ser. A Mat. RACSAM},
  115(3):Paper No. 148, 13, 2021.

\bibitem{GDMP22}
M.~V. Hern\'{a}ndez, M.~B. Lattanzi, and N.~Thome.
\newblock {GDMP}-inverses of a matrix and their duals.
\newblock {\em Linear Multilinear Algebra}, 2022.

\bibitem{mp2}
M.~V. Hern\'{a}ndez, M.~B. Lattanzi, and N.~Thome.
\newblock On 2{MP}-, {MP}2-and {C2MP}-inverses for rectangular matrices.
\newblock {\em Rev. R. Acad. Cienc. Exactas F\'{\i}s. Nat. Ser. A Mat. RACSAM},
  116(4):Paper No. 156, 2022.

\bibitem{king1977note}
C.~F. King.
\newblock A note on {D}razin inverses.
\newblock {\em Pacific J. Math.}, 70(2):383--390, 1977.

\bibitem{lay1975spectral}
D.~C. Lay.
\newblock Spectral properties of generalized inverses of linear operators.
\newblock {\em SIAM Journal on Applied Mathematics}, 29(1):103--109, 1975.

\bibitem{ma2021characterizations}
H.~Ma.
\newblock Characterizations and representations for the {CMP} inverse and its
  application.
\newblock {\em Linear and Multilinear Algebra}, pages 1--16, 2021.

\bibitem{MaGaSt}
H.~Ma, X.~Gao, and P.~S. Stanimirovi\'{c}.
\newblock Characterizations, iterative method, sign pattern and perturbation
  analysis for the {DMP} inverse with its applications.
\newblock {\em Appl. Math. Comput.}, 378:125196, 18, 2020.

\bibitem{malik2014new}
S.~B. Malik and N.~Thome.
\newblock On a new generalized inverse for matrices of an arbitrary index.
\newblock {\em Appl. Math. Comput.}, 226:575--580, 2014.

\bibitem{manjunatha2014core}
K.~Manjunatha~Prasad and K.~S. Mohana.
\newblock Core-{EP} inverse.
\newblock {\em Linear Multilinear Algebra}, 62(6):792--802, 2014.

\bibitem{MahSal}
M.~Mehdipour and A.~Salemi.
\newblock On a new generalized inverse of matrices.
\newblock {\em Linear Multilinear Algebra}, 66(5):1046--1053, 2018.

\bibitem{meng2017dmp}
L.~Meng.
\newblock The {DMP} inverse for rectangular matrices.
\newblock {\em Filomat}, 31(19):6015--6019, 2017.

\bibitem{core-nil}
S.~K. Mitra, P.~Bhimasankaram, and S.~B. Malik.
\newblock {\em Matrix partial orders, shorted operators and applications},
  volume~10 of {\em Series in Algebra}.
\newblock World Scientific Publishing Co. Pte. Ltd., Hackensack, NJ, 2010.

\bibitem{mosic2018weighted}
D.~Mosi\'{c}.
\newblock Weighted g{DMP} inverse of operators between {H}ilbert spaces.
\newblock {\em Bull. Korean Math. Soc.}, 55(4):1263--1271, 2018.

\bibitem{mosic2019weighted}
D.~Mosi\'{c}.
\newblock Weighted {G}-{D}razin inverse for operators on {B}anach spaces.
\newblock {\em Carpathian J. Math.}, 35(2):171--184, 2019.

\bibitem{mosic2018gdmp}
D.~Mosi\'{c} and D.~S. Djordjevi\'{c}.
\newblock The g{DMP} inverse of {H}ilbert space operators.
\newblock {\em J. Spectr. Theory}, 8(2):555--573, 2018.

\bibitem{nashed1992drazin}
M.~Z. Nashed and Y.~G. Zhao.
\newblock The {D}razin inverse for singular evolution equations and partial
  differential operators.
\newblock In {\em Recent trends in differential equations}, volume~1 of {\em
  World Sci. Ser. Appl. Anal.}, pages 441--456. World Sci. Publ., River Edge,
  NJ, 1992.

\bibitem{D1}
J.~K. Sahoo, G.~Maharana, B.~Sitha, and N.~Thome.
\newblock 1{D} inverse and {D}1 inverse of square matrices.
\newblock {\em To appear in Miskolc Math. Notes}.

\bibitem{chen2019representation}
M.~Sheibani~Abdolyousefi.
\newblock The representations of the {G}-{D}razin inverse in a {B}anach algebra.
\newblock {\em Hacet. J. Math. Stat.}, 50(3):659--667, 2021.

\bibitem{simeon1993drazin}
B.~Simeon, C.~F\"{u}hrer, and P.~Rentrop.
\newblock The {D}razin inverse in multibody system dynamics.
\newblock {\em Numer. Math.}, 64(4):521--539, 1993.

\bibitem{wang2016partial}
H.~Wang and X.~Liu.
\newblock Partial orders based on core-nilpotent decomposition.
\newblock {\em Linear Algebra Appl.}, 488:235--248, 2016.
\end{thebibliography}
\end{document}